\documentclass[leqno]{article}
\usepackage{amsmath}
\usepackage{amssymb}
\usepackage{amsthm}
\allowdisplaybreaks
\title{A study on the bilinear equation
 of the sixth Painlev\'e transcendents}

\author{Tatsuya \textsc{Hosoi}\footnote{Graduate School of Mathematical Sciences,
University of Tokyo, Tokyo
153-8914, Japan.\newline e-mail: \texttt{hosoi@ms.u-tokyo.ac.jp}}
          ~and Hidetaka \textsc{Sakai}\footnote{Graduate School of Mathematical Sciences,
University of Tokyo, Tokyo
153-8914, Japan. \endgraf e-mail: \texttt{sakai@ms.u-tokyo.ac.jp}}}

\begin{document}

\maketitle
\newtheorem{thm}{Theorem}[section]
\newtheorem{prop}[thm]{Proposition}
\def\pf{\noindent{\it Proof.\quad}}
\def\qed{\hfil\fbox{}\medskip}
\theoremstyle{remark}
\newtheorem{rem}{\bf {\slshape Remark}}[section]
\numberwithin{equation}{section}

\begin{abstract}

The sixth Painlev\'e equation is a basic equation among the non-linear differential equations with three fixed singularities, corresponding to Gauss's hypergeometric differential equation among the linear differential equations.
It is known that 2nd order Fuchsian differential equations with three singular points are reduced to the hypergeometric differential equations.
Similarly, for nonlinear differential equations, we would like to determine the equation from the local behavior around the three singularities.
In this paper, the sixth Painlev\'e equation is derived by imposing the condition that it is of type (H) at each three singular points for the homogeneous quadratic 4th-order differential equation.



\end{abstract}

\section{Introduction}

The Painlev\'e equations are second-order nonlinear ordinary differential equations proposed in the motive of searching for new special functions following the success of research on special functions like elliptic functions and hypergeometric functions in the 19th century.

Two perspectives are particularly important in the study of the Painlev\'e equations.
One is the classification of second-order algebraic differential equations with Painlev\'e properties by Painlev\'e and Gambier \cite{Pn1, Pn2, Gm}.
The other is R. Fuchs's deformation theory of linear differential equations that preserves monodromy.
In particular, the sixth Painlev\'{e} equation is obtained from the deformation theory of the 2nd-order Fuchsian equation with 4 singular points \cite{F}.

The sixth Painlev\'e equation dealt with in this paper is the most generic one of the Painlev\'e equations.
It is known that all the other Painlev\'e equations can be obtained by degenerating the sixth equation.

The sixth Painlev\'e equation has a slightly complicated form and is written as
\begin{align} \label{eqn:p6}
 \frac{d^2q}{dt^2}=&\ \frac{1}{2} \left( \frac{1}{q} +\frac{1}{q-1}+\frac{1}{q-t} \right) \left( \frac{dq}{dt} \right)^2- \left( \frac{1}{t}+ \frac{1}{t-1}+ \frac{1}{q-t} \right) \frac{dq}{dt} \\
&+\frac{q(q-1)(q-t)}{t^2(t-1)^2} \left( a +\frac{b t}{q^2}+\frac{c(t-1)}{(q-1)^2} +\frac{d t(t-1)}{(q-t)^2}\right)\nonumber .
\end{align}
Here $a$, $b$, $c$, $d$ are parameters.
However,
it takes time and effort to derive the concrete expression of this equation
starting from the Painlev\'{e} property or the isomonodromic deformation theory.
The classification of differential equations that satisfy the Painlev\'e property is very complicated.
In Gambier's paper, the argument is to eliminate equations that do not satisfy the Painlev\'e property, so the sixth Painlev\'e equation cannot be derived without a full classification.
The isomonodromic deformation of the four-point Fuchsian equation is clear, but it still needs long calculation.

What we would like to do in this paper is to determine simply the concrete expression of the equation from a local condition around the singularity $t = 0,1, \infty$.
However, in detail, what we characterize is not the 2nd-order equation of normal type \eqref{eqn:p6}, but an equivalent 4th-order equation called a bilinear equation.
Just as hypergeometric differential equations can be characterized as second-order linear equations with three regular singularities, we would like to characterize the sixth Painlev\'{e} equation as a fourth-order homogeneous quadratic equation with three singularities of good nature.
Here, ``homogeneous quadratic'' means that the equation does not contain terms other than quadratic terms with respect to the unknown function $f$ and its derivatives.

Our results can be stated as follows: A homogeneous quadratic 4th-order equations of type (H) with 3 singular points is determined as
\begin{align}
 0 =&\left[ (t-1)^4{{\cal D}_{\log t}}^4
 +2(t-1)^3(t+1)\delta {{\cal D}_{\log t}}^2+t(t-1)^2\delta^2
 \right.\\
 &+(t-1)^2(t(t-1)+1-\beta_1-\beta_5t-\beta_6t^2)
 {{\cal D}_{\log t}}^2\nonumber\\
 &\left. +t(t-1)(\beta_4(t-1)-\beta_1-\beta_5-\beta_6)\delta
 +t(t-1)(\beta_2t+\beta_3)
 \right] f\cdot f,\nonumber
\end{align}
and gauge transformations reduce it to the case that $\beta_5 = \beta_6 = 0$.
Furthermore, when $\beta_4 =-\frac12\beta_5- \beta_6$, it is reduced to a second-order differential equation, which coincides with the bilinear form of the sixth Painlev\'e equation. The differential operator $\delta=t\frac{d}{dt}$, and ${\cal D}_{\log t}$ is the Hirota derivative for $\delta$; see (\ref{Hirota}).

The definition of type (H) is given at the end of Section 3 and the beginning of Section 4. We say that a homogeneous quadratic differential equation is of type (H) at the singular point $t=0$, when the lowest term for the independent variable $t$ of the equation is
a constant multiple of
\begin{align}
&2t^4\left\{ f\frac{d^4f}{dt^4}-4\frac{df}{dt}\frac{d^3f}{dt^3}+3\left( \frac{d^2f}{dt^2} \right)^2 \right\}+8t^3\left( f\frac{d^3f}{dt^3} -\frac{df}{dt}\frac{d^2f}{dt^2} \right)+4t^2f\frac{d^2f}{dt^2} \label{eq+} \\
&+8\alpha \left\{ t^2\left( f \frac{d^2f}{dt^2} -\left(\frac{df}{dt}\right)^2 \right) +tf\frac{df}{dt} \right\}. \nonumber
\end{align}
with a constant $\alpha$.
The definition of type (H) is based on the existence and the convergence of formal series solutions at the singular point.

In Section 4, we look at the explicit form of the equation whose all 3 fixed singular points $t=0,1,\infty$ are of type (H)


\section{Bilinear equation of the sixth Painlev\'{e} transcendents}

First, let us review the bilinear form of the sixth Painlev\'e equation (cf. \cite{O, JM}).
The sixth Painlev\'e equation can be expressed as a non-autonomous Hamilton system using the following Hamiltonian:
\begin{align}
 t(t-1)H\left({\theta_0, \theta_1, \theta_t\atop
 \kappa_1,~\kappa_2
 };t;q,p\right)
 =&\ q(q-1)(q-s)p^2\\
 &+\left\{\frac{\theta_t}{q}+\frac{\kappa_1-\kappa_2-1}{q-1}+\frac{\theta_0+1}{q-t}\right\}
 q(q-1)(q-t)p\nonumber\\
 &+\kappa_2(\theta_1+\kappa_2)q-t(\theta_t+\kappa_1)\theta_t
 -\theta_0\theta_t.\nonumber
\end{align}
Here
$\theta_0$, $\theta_1$, $\theta_t$,
$\kappa_1$, $\kappa_2$
are parameters satisfying
$\theta_0+\theta_1+\theta_t+\kappa_1+\kappa_2=0$.
The correspondence with the expression (\ref{eqn:p6}) is as follows:
\begin{align}
 &a =\frac{{\theta_1}^2}{2},\quad
 b =-\frac{{\theta_t}^2}{2},\quad
 c =\frac{(\kappa_1-\kappa_2-1)^2}{2},\quad
 d =\frac{1-{\theta_0}^2}{2}.
\end{align}

The system can be solved by finding $q$, $p$ as a function of $t$,
but instead, we may find the Hamiltonian function $h(t)$ substituting $q(t) $, $p(t)$,
because the functions $q$ and $p$ are written in rational expressions
of $h$ and its derivatives.
In fact,
when we put
\begin{align}
 h(t)&=t(t-1)H\left(q(t),p(t)\right)-\kappa_2 (\theta_1 +\kappa_2)t
 +t(\theta_t+\kappa_1)\theta_t+\theta_0\theta_t\\
 &=q(q-1)(q-t)p^2
 +\left\{\frac{\theta_t}{q}+\frac{\kappa_1-\kappa_2-1}{q-1}+\frac{\theta_0+1}{q-t}\right\}
 q(q-1)(q-t)p
 \nonumber\\
 &\hspace*{3em}+\kappa_2 (\theta_1 +\kappa_2)(q-t),\nonumber
\end{align}
we get
\begin{align}
 p=&\ \frac{h+(q-t)\frac{dh}{dt}}{(\theta_0+1)q(q-1)},\\
 q=&-\frac12\frac{N}{D},\\
 &N=t(t-1)\frac{d^2h}{dt^2}
 -\left(t\frac{dh}{dt}-h\right)
 \left(\frac{2}{\theta_0+1}\frac{dh}{dt}+\theta_t+\kappa_1-\kappa_2-1\right)\nonumber\\
 &\hspace*{3em}
 -(\theta_0+\theta_t+1)\frac{dh}{dt}-(\theta_0+1)\kappa_2(\theta_1+\kappa_2),\nonumber\\
 &D=\frac{dh}{dt}\left(\frac{1}{\theta_0+1}\frac{dh}{dt}+\theta_0+\theta_t+\kappa_1-\kappa_2\right)
 +(\theta_0+1)\kappa_2(\theta_1+\kappa_2)\nonumber.
\end{align}

Now, we know that the Hamiltonian function $h$ satisfies the second-order equation:
\begin{align}
 &\left(t(t-1)\frac{d^2h}{dt^2}\right)^2=-\left(4\frac{dh}{dt}+4\kappa_2
 (\theta_1+\kappa_2)-(\theta_t +\kappa_1-\kappa_2-1)^2\right)
 \left(t\frac{dh}{dt}-h\right)^2\\
 &+4\left(\!\left(\frac{dh}{dt}\right)^2
 +\frac{(\theta_0-\theta_t+1)(\kappa_1-\kappa_2-1)
 -\theta_t(\theta_0+\theta_t+1)+2\kappa_2(\theta_1+\kappa_2)}{2}
 \frac{dh}{dt}\right.\nonumber\\
 &\qquad\left. +\frac{(\theta_0+1)\kappa_2(\theta_1+\kappa_2)
 \left(\kappa_1-\kappa_2-\theta_t -1\right)}{2} \right)
 \left(t\frac{dh}{dt}-h\right)\nonumber \\
 &+\left((\theta_0 +\theta_t+1)\frac{dh}{dt}+(\theta_0+1)\kappa_2 (\theta_1+\kappa_2)\right)^2.\nonumber
\end{align}
Once the solution $h$ of this equation is found, then $q$ and $p$ are obtained, and the sixth Painlev\'e equation is solved.

Now, let us consider rewriting this equation into simpler one with the Hirota derivative.
Differentiation again gives the 3rd-order differential equation
\begin{align}
 0=&t^2(t-1)^2\frac{d^3h}{dt^3}+(2t-1)t(t-1)\frac{d^2h}{dt^2}
 +6t(t-1)\left(\frac{dh}{dt}\right)^2
 -4(2t-1)h\frac{dh}{dt}+2h^2\\
 &-(\beta_1+t\beta_5+t^2\beta_6)\frac{dh}{dt}
 +(\beta_4+\beta_5+(t+1)\beta_6)h+\frac12 (\beta_2t+\beta_3).\nonumber
\end{align}
Here
\begin{align}
  &\beta_1=(\theta_0+\theta_1+1)^2,\quad
 \beta_2=2(\theta_0+1)\kappa_2(\theta_1+\kappa_2)(\kappa_2-\kappa_1+\theta_t+1),
 \\
 &\beta_3=-2(\theta_0+\theta_t+1)(\theta_0+1)\kappa_2(\theta_1+\kappa_2),\qquad
 \beta_6=(\theta_t+\kappa_1-\kappa_2-1)^2-4\kappa_2(\theta_1+\kappa_2),\nonumber\\
 &\beta_5/2=(\theta_0+\theta_t+1)(\theta_t+\kappa_1-\kappa_2-1)
 +\kappa_2(\theta_1+\kappa_2),\qquad
 \beta_4=-\frac12\beta_5-\beta_6.\nonumber
\end{align}
Furthermore, putting
$h=t(t-1)\frac{d}{dt}\log f$,
the equation is rewritten as
\begin{align}\label{eqn:p6_bilin}
 0=&\
 \left[ (t-1)^4{{\cal D}_{\log t}}^4
 +2(t-1)^3(t+1)\delta {{\cal D}_{\log t}}^2+t(t-1)^2\delta^2
 \right.\\
 &+(t-1)^2(t(t-1)+1-\beta_1-\beta_5t-\beta_6t^2)
 {{\cal D}_{\log t}}^2\nonumber\\
 &\left. +t(t-1)(\beta_4(t-1)-\beta_1-\beta_5-\beta_6)\delta
 +t(t-1)(\beta_2t+\beta_3)
 \right] f\cdot f. \nonumber
\end{align}
The differential operator $\delta = t\frac{d}{dt}$ is used here.
$\mathcal{D}_{\log t}$ is the Hirota derivative for $\delta$
and we define it as
\begin{equation}
\mathcal{D}^N_{\log t} f\cdot g:=\sum_{i=0}^N (-1)^i \binom{N}{i} \left( \delta^{N-i}f\right) \left( \delta^ig \right). \label{Hirota}
\end{equation}
The Hirota derivative is an operator that gives one function from a pair of functions.

The equation (\ref{eqn:p6_bilin}) is called the bilinear form of the sixth Painlev\'e equation.
Since it is written by using a bilinear operator, it is often called a bilinear form similarly to the case of soliton equations.
However, in this case of the Painlev\'e equation, it is actually a homogeneous quadratic relation rather than a bilinear form.
Note that the quadratic equation of the sixth Painlev\'e equation satisfies $2\beta_4 + \beta_5 + 2\beta_6 = 0$, which is the condition that the equation is reduced to an equation of second order. (See the end of the section 4.)

The solution of the equation (\ref{eqn:p6_bilin}) is called the sixth Painlev\'e $\tau$-function.

\section{Series solution}

We have seen that the $\tau$-function of the sixth Painlev\'{e} equation satisfies the homogeneous quadratic differential equation (\ref{eqn:p6_bilin}) with $2\beta_4 + \beta_5 + 2\beta_6 = 0$.
A series solution of the equation is also known \cite{GIL, ILT, BS, GL}.
The solution is written in the form
\begin{equation}
\tau (t)=\sum_{k\in \mathbb{Z},l\in \mathbb{Z}_{\ge 0}} b_{k,l}(\sigma) t^{(\sigma +k)^2+l}. \label{eqe}
\end{equation}
Here, $\sigma$ is a characteristic parameter that appears in powers of the series, and it is an arbitrary complex parameter independent of the equation's parameters $a$, $b$, $c$, $d$. The parameter $\sigma$ corresponds to initial values.
Also, $b_{k,l}(\sigma)$'s are coefficients of the series and do not depend on the variable $t$.

This series is also a special case of a series of the form
\begin{equation}
f(t)=t^{\sigma^2}\sum_{\substack{(m,n)\in {\mathbb{Z}_{\ge 0}}^2 \\ \text{or} \ (m,n)\in \left( \frac{1}{2} +\mathbb{Z}_{\ge 0} \right)^2}} a_{m,n}(\sigma)t^{(2\sigma +1)m+(-2\sigma +1)n}\label{eqa}
\end{equation}
Recently, one of the authors has shown that the following two theorems hold for a homogeneous quadratic differential equation with a series solution of this form \cite{H}:

\begin{thm}
Let $N_0$ be a positive integer.
A homogeneous quadratic differential equation
 \begin{align}
&\sum_{0 \le K_1 \le K_2 \le 4}\alpha_{0,K_1,K_2}t^{K_1+K_2} \frac{d^{K_1}f}{dt^{K_1}}\frac{d^{K_2}f}{dt^{K_2}}\\
&+\sum_{N=1}^{N_0} t^N \sum_{0 \le K_1 \le K_2< \infty}\alpha_{N,K_1,K_2}t^{K_1+K_2} \frac{d^{K_1}f}{dt^{K_1}}\frac{d^{K_2}f}{dt^{K_2}}
=0 \label{eqb}\nonumber
\end{align}
has a general series solution of the form \eqref{eqa}
which has arbitrary parameters $\sigma \in \mathbb{C}\setminus\mathbb{Q}$
and $a_{0,0},a_{1,0},a_{0,1}\in\mathbb{C}\setminus\{0\}$,
if and only if
the sum of the lowest-degree terms
$\displaystyle\sum_{0 \le K_1 \le K_2 \le
 4}\!\!\!\!\alpha_{0,K_1,K_2}t^{K_1+K_2}
 \frac{d^{K_1}f}{dt^{K_1}}\frac{d^{K_2}f}{dt^{K_2}}$
is a constant multiple of
\begin{align}
 \left[\mathcal{D}_{\log t}^4+(1-2\delta)\mathcal{D}^2_{\log t} \right]\, f\cdot f.
\end{align}
Here, $\delta = t \frac{d}{dt}$, and the degree is with respect to $t$ \rm{(}$\mathrm{deg} \, t =1$\rm{)} while the degree of $d / dt$ is counted as $-1$.
\end{thm}
In particular, 3rd-order (or lower order) homogeneous quadratic differential equations do not have
the series solution of the form (\ref{eqa}).

\begin{thm}
 Let $N_0$ be a positive integer.
If a 4th-order homogeneous quadratic differential equation
\begin{equation}
\left[\mathcal{D}_{\log t}^4+(1-2\delta)\mathcal{D}^2_{\log t} \right]\, f\cdot f +\sum_{N=1}^{N_0} t^N \sum_{K=0}^{4} \sum_{I=0}^{\left[ \frac{1}{2} K \right]}\alpha_{N,K,I}\delta^I f \delta^{K-I}f=0 \label{eq1}
\end{equation}
has a formal series solution
\begin{equation}
f(t)=t^{\sigma^2}\sum_{\substack{(m,n)\in {\mathbb{Z}_{\ge 0}}^2 \\ \text{or} \ (m,n)\in \left( \frac{1}{2} +\mathbb{Z}_{\ge 0} \right)^2}} a_{m,n}t^{(2\sigma +1)m+(-2\sigma +1)n}\label{eq2}
\end{equation}
with $\sigma \notin \mathbb{Q},\left| \mathrm{Re} \ \sigma \right| <\frac{1}{2}$,
then the series has a convergence region of positive radius in any angular region of $\mathbb{C} \setminus \{0 \}$.
\end{thm}

Here, the terms $\left[\mathcal{D}_{\log t}^4+(1-2\delta)\mathcal{D}^2_{\log t} \right]\, f\cdot f$
are written as
\begin{equation}
2t^4\left\{ f\frac{d^4f}{dt^4}-4\frac{df}{dt}\frac{d^3f}{dt^3}+3\left( \frac{d^2f}{dt^2} \right)^2 \right\}+8t^3\left( f\frac{d^3t}{dt^3} -\frac{df}{dt}\frac{d^2f}{dt^2} \right)+4t^2f\frac{d^2f}{dt^2}. \label{eq+}
\end{equation}
We can loosen the existence condition of series solutions of the form (\ref{eqa}) to the existence condition of series solution of the form
\begin{align}
 t^{\alpha +\sigma^2}\sum_{\substack{(m,n)\in {\mathbb{Z}_{\ge 0}}^2 \\ \text{or} \ (m,n)\in \left( \frac{1}{2} +\mathbb{Z}_{\ge 0} \right)^2}} a_{m,n}(\sigma)t^{(2\sigma +1)m+(-2\sigma +1)n} \label{gauge}
\end{align}
by gauge transformation $f \mapsto t^\alpha f$ and equation
\begin{align}
 \left[\mathcal{D}_{\log t}^4+(1-2\delta)\mathcal{D}^2_{\log t} \right]\,
t^\alpha f\cdot t^\alpha f=t^{2\alpha}\left[\mathcal{D}_{\log t}^4+(1-4\alpha -2\delta)\mathcal{D}^2_{\log t} \right]\, f\cdot f.
\end{align}

Let us say that a homogeneous quadratic differential equation is of type
(H) at the singular point $t=0$,
when the lowest term with respect to the independent variable $t$ of the equation is
a constant multiple of
\begin{align}
 \left[\mathcal{D}_{\log t}^4+(1+4\alpha -2\delta)\mathcal{D}^2_{\log t} \right]\,
 f\cdot f
\end{align}
with a constant $\alpha$.
In this case we have a series solution of the form (\ref{gauge})
near the singular point $t=0$.
We define singularity of type (H) at $t=t_0$ by taking $s=t-t_0$ (when $t_0 = \infty$, $s=1/t$).

We can easily check out that the three fixed singular points of the sixth Painlev\'{e} equation of bilinear form are of type (H). Furthermore, ones of the two fixed singular points of the third and fifth Painlev\'{e} equations of bilinear form are of type (H).
For the first, second and fourth Painlev\'e equations, the singular points of the equations of bilinear form are not of type (H).

\section{Homogeneous quadratic differential equation of type (H) with 3 singular points}

We consider
\begin{quote}
($*$) \quad fourth-order homogeneous quadratic differential equations that have only terms $\alpha_{k, l} (t)\left( \frac{d^k}{dt^k}f \right) \left( \frac{d^l}{dt^l} f\right)$ for which $k + l \leq 4$ holds.
\end{quote}
Of the equations ($*$), an equation whose all the fixed singular points are
of type (H) is called {\bf an equation of type (H)}.
When we say an equation of type (H), we assume the condition ($*$).

Here, the fixed singular point is the zero of the coefficient of the highest order term (in our case, the 4th-order term).
We say that $t=\infty$ is a fixed singular point when $s=0$ is a fixed singular point changing variable $t$ into $s=1/t$.

\begin{thm}
 An equation of type (H) with fixed singular points only at $t = 0,1,\infty$ on $\mathbb{P}^1$ can be written as $0={\cal L}\, f\cdot f$, where
\begin{align}
{\cal L}
 =& (t-1)^4{{\cal D}_{\log t}}^4
 +2(t-1)^3(t+1)\delta {{\cal D}_{\log t}}^2+t(t-1)^2\delta^2
 \\
 &+(t-1)^2(t(t-1)+1-\beta_1-\beta_5t-\beta_6t^2)
 {{\cal D}_{\log t}}^2\nonumber\\
 & +t(t-1)(\beta_4(t-1)-\beta_1-\beta_5-\beta_6)\delta
 +t(t-1)(\beta_2t+\beta_3).\nonumber
\end{align}
\end{thm}

\begin{proof}
 Let us express the equation in $t$, $\delta = t\, d/dt$ instead of $t$ and $d/dt$.
 Let $n$ be the degree of $\mathcal{L}$.
Since $t = 0, \infty$ is of type (H), the equation can be written as
\begin{align}
 0=&\left[{{\cal D}_{\log t}}^4+\left( 1+\alpha_1 -2\delta\right){{\cal D}_{\log t}}^2
 \right] f\cdot f\\
 &+t\sum_{0\leq l\leq k\leq 4\atop k+l\leq 4}a_{k,l}(t)\delta^k f\delta^l f
 +t^n\left[{{\cal D}_{\log t}}^4+\left( 1+\alpha_2 +2\delta\right){{\cal D}_{\log t}}^2
 \right] f\cdot f.\nonumber
\end{align}
Here, $a_{k,l}$ are polynomials of degree $n-2$.
Note that when we set $s=1/t$,
\begin{equation}
\left[ {{\cal D}_{\log t}}^4+\left( 1+\alpha_2 +2\delta\right){{\cal D}_{\log t}}^2 \right]f \cdot f=\left[ {{\cal D}_{\log s}}^4+\left( 1+\alpha_2 -2s\frac{d}{ds}\right){{\cal D}_{\log s}}^2 \right] f\cdot f
\end{equation}

Since the fixed singular point of the equation is only $t = 1$ except for $t = 0$ and $\infty$,
the coefficient of $f\delta^4 f$ in $\mathcal{L} f \cdot f$ is $2(1+t^n)+ta_{40}(t)=2(t-1)^n$.

Setting $u=t-1$, $\delta_u=ud/du$, we have
$\delta =(u+1)\frac{d}{du}=\frac{u+1}{u}\delta_u$, and
\begin{align}
 \delta^2
 =&(u+1)\frac{d}{du}(u+1)\frac{d}{du}=(u+1)^2\frac{d^2}{du^2}+(u+1)\frac{d}{du}\\
 =&\frac{(u+1)^2}{u^2}\delta_u(\delta_u-1)+\frac{u+1}{u}\delta_u,\nonumber\\
 \delta^3
 =&(u+1)^3\frac{d^3}{du^3}+3(u+1)^2\frac{d^2}{du^2}+(u+1)\frac{d}{du}\nonumber\\
 =&\frac{(u+1)^3}{u^3}\delta_u(\delta_u-1)(\delta_u-2)
 +\frac{3(u+1)^2}{u^2}{\delta_u}(\delta_u-1)+\frac{u+1}{u}\delta_u,\nonumber\\
 \delta^4
 =&(u+1)^4\frac{d^4}{du^4}+6(u+1)^3\frac{d^3}{du^3}+7(u+1)^2\frac{d^2}{du^2}
 +(u+1)\frac{d}{du}\nonumber\\
 =&\frac{(u+1)^4}{u^4}\delta_u(\delta_u-1)(\delta_u-2)(\delta_u-3)
 +\frac{6(u+1)^3}{u^3}\delta_u(\delta_u-1)(\delta_u-2)\nonumber\\
 &+\frac{7(u+1)^2}{u^2}{\delta_u}(\delta_u-1)+\frac{u+1}{u}\delta_u.\nonumber
\end{align}
Since the coefficient of $f\delta^4f$ in $\mathcal{L}f\cdot f$ is $2(t-1)^n=2u^n$,
the lowest degree in $u$ is $n-4$.
This shows that the whole equation $\mathcal{L} f \cdot f =0$ is divisible by $u^{n-4}$ when expressed by $u$,
because $u=0$ is of type (H).
If we divide it, we can reduce it into the case $n = 4$.
(If we take $v=-u$ as a variable, the singular points of the equation
 become $v = 0, 1,\infty$ again.)

Assuming $n = 4$, the lowest degree is 0, and the
 coefficient of $\delta^kf\delta^lf $ must be divisible by $u^{l + k}$.
Therefore,
the term with $k+l=4$ in $\mathcal{L} f\cdot f$ is written as $(t-1)^4{{\cal D}_{\log t}}^4f\cdot f$,
the term with $k+l=3$ is $$2(t-1)^3(t+1)\delta {{\cal D}_{\log t}}^2f\cdot f,$$
the term with $k+l=2$ is $$(t-1)^2\left\{(1+\alpha_1+t^2+\alpha_2 t^2){{\cal D}_{\log t}}^2f\cdot f
+\alpha_3 tf\delta^2f+\alpha_4 t(\delta f)^2\right\},$$
the term with $k+l=1$ is $t(t-1)(\alpha_5t+\alpha_6)f\delta f$,
and the term with $k+l=0$ is written as
 $t(\alpha_7t^2+\alpha_8t+\alpha_9)f^2$.

Here, if the 0th-degree term in $u$ is of type (H), the condition is satisfied.
The equation is written as
\begin{align}
 0={\cal L}f\cdot f
 =&\left[ (t-1)^4{{\cal D}_{\log t}}^4
 +2(t-1)^3(t+1)\delta {{\cal D}_{\log t}}^2\right.\\
 &+(t-1)^2\left\{(1+\alpha_1+\frac{\alpha_3-\alpha_4}{4}t+(1+\alpha_2)t^2)
 {{\cal D}_{\log t}}^2 +\frac{\alpha_3+\alpha_4}{4}t\delta^2\right\}
 \nonumber\\
 &\left. +\frac12 t(t-1)(\alpha_5t+\alpha_6)\delta
 +t(\alpha_7t^2+\alpha_8t+\alpha_9)
 \right] f\cdot f,\nonumber
\end{align}
so the 0th-degree term in $u$ is
\begin{align}\label{eqn:0th_term}
 &2f\delta_u(\delta_u-1)(\delta_u-2)(\delta_u-3)f
 -8\delta_uf \delta_u(\delta_u-1)(\delta_u-2)f\\
 &+6\{\delta_u(\delta_u-1)f\}^2
 +8f\delta_u(\delta_u-1)(\delta_u-2)f
 -8\delta_uf\delta_u(\delta_u-1)f\nonumber\\
 &+2(2+\alpha_1+\alpha_2+\frac{\alpha_3}{2})f\delta_u(\delta_u-1)f
 -2(2+\alpha_1+\alpha_2-\frac{\alpha_4}{2})(\delta_uf)^2\nonumber\\
 &+(\alpha_5+\alpha_6)f\delta_uf
 +(\alpha_7+\alpha_8+\alpha_9)f^2.\nonumber
\end{align}
Let us see each term of $\delta^kf\delta^lf$, $k + l = 4,3,2,1,0$
in the differential polynomial \eqref{eqn:0th_term}.
When $k+l=4$, it is
\begin{align}
 &2f{\delta_u}^4f-8\delta_uf{\delta_u}^3f+6({\delta_u}^2f)^2
 ={{\cal D}_{\log u}}^4f\cdot f.
\end{align}
When $k+l=3$, it is
\begin{align}
 &-12f{\delta_u}^3f+24\delta_uf{\delta_u}^2f
 -12\delta_uf{\delta_u}^2f+8f{\delta_u}^3f
 -8\delta_uf{\delta_u}^2f\\
 =&-4(f{\delta_u}^3f-\delta_uf{\delta_u}^2f)
 =-2\delta_u{{\cal D}_{\log u}}^2f\cdot f.\nonumber
\end{align}
When $k+l=2$, it is
\begin{align}
 &(2+2\alpha_1+2\alpha_2+\alpha_3)f{\delta_u}^2f
 -(6+2\alpha_1+2\alpha_2-\alpha_4)(\delta_uf)^2,
\end{align}
and we need
\begin{align}
 &\alpha_4=-\alpha_3+4.
\end{align}
The coefficeint of $f\delta_uf$ is
$-12+16-4-2\alpha_1-2\alpha_2-\alpha_3+\alpha_5+\alpha_6$.
It is necessary that
\begin{align}
 &\alpha_6=2\alpha_1+2\alpha_2+\alpha_3-\alpha_5.
\end{align}
Furthermore, from the coefficient of $f^2$,
we have
\begin{align}
 &\alpha_7+\alpha_8+\alpha_9=0.
\end{align}
After all, the six free parameters $\alpha_1$, $\alpha_2$, $\alpha_3$, $\alpha_5$,
$\alpha_7$, $\alpha_8$
remain and we obtain
\begin{align}
 0={\cal L}f\cdot f
 =&\left[ (t-1)^4{{\cal D}_{\log t}}^4
 +2(t-1)^3(t+1)\delta {{\cal D}_{\log t}}^2\right.\\
 &+(t-1)^2(1+\alpha_1-(1-\tilde{\alpha}_3)t+(1+\alpha_2)t^2)
 {{\cal D}_{\log t}}^2 +t(t-1)^2\delta^2
 \nonumber\\
 &\left. +t(t-1)(\tilde{\alpha}_5(t-1)+\alpha_1+\alpha_2+\tilde{\alpha}_3)\delta
 +t(t-1)(\tilde{\alpha}_7(t-1)+\tilde{\alpha}_8)
 \right] f\cdot f.\nonumber
\end{align}
\end{proof}

Let us look at the relationship between this equation of type (H) with
three singular points and the sixth Painlev\'e equation.

First, the equation of type (H) with 3 singular points is converted as
\begin{align}
 0={\cal L}f\cdot f
 =&\left[ (t-1)^4{{\cal D}_{\log t}}^4
 +2(t-1)^3(t+1)\delta {{\cal D}_{\log t}}^2+t(t-1)^2\delta^2
 \right.\\
 &+(t-1)^2(t(t-1)+1-\beta_1-\beta_5t-\beta_6t^2)
 {{\cal D}_{\log t}}^2\nonumber\\
 &\left. +t(t-1)(\beta_4(t-1)-\beta_1-\beta_5-\beta_6)\delta
 +t(t-1)(\beta_2t+\beta_3)
 \right] f\cdot f\nonumber\\
 =&t^{2\alpha}(t-1)^{2\gamma}\left[ (t-1)^4{{\cal D}_{\log t}}^4
 +2(t-1)^3(t+1)\delta {{\cal D}_{\log t}}^2
 +t(t-1)^2\delta^2\right.\nonumber\\
 &+(t-1)^2(t(t-1)+1-\beta_1-2\alpha-(\beta_5+4\gamma)t+(-\beta_6+2\alpha+2\gamma)t^2)
 {{\cal D}_{\log t}}^2 \nonumber\\
 &+t(t-1)((\beta_4+2\alpha)(t-1)-\beta_1-\beta_5-\beta_6-2\gamma)\delta
 \nonumber\\
 &+t(t-1)((\beta_2+\alpha^2+\alpha\beta_4
 -\gamma^2+\gamma (\beta_4+\beta_6))t\nonumber\\
 &\left. +\beta_3-\alpha(\alpha+\beta_1+\beta_4+\beta_5+\beta_6)
 -\gamma\beta_1-2\alpha\gamma)
 \right] g\cdot g\nonumber
\end{align}
by the gauge transformation $f = t^\alpha (t-1)^\gamma g$.
With this transformation, it can be reduced to the case $\beta_5 = \beta_6 = 0$.

The bilinear form of the sixth Painlev\'e equation \eqref{eqn:p6_bilin}
is reduced to the case
\begin{align}
 &\beta_1=\frac{(\theta_0+1)^2+{\theta_1}^2+{\theta_t}^2+(\kappa_1-\kappa_2-1)^2}{2},\\
 &\beta_2=\frac{(\theta_0+\theta_1+1)(\theta_0-\theta_1+1)
 (\theta_t+\kappa_1-\kappa_2-1)(\theta_t-\kappa_1+\kappa_2+1)}{8},\nonumber\\
 &\beta_3=\frac{((\theta_0+1)^2-{\theta_1}^2-{\theta_t}^2+(\kappa_1-\kappa_2-1)^2)^2
 +4((\theta_0+1)^2+{\theta_1}^2)({\theta_t}^2+(\kappa_1-\kappa_2-1)^2)}{32},\nonumber\\
 &\beta_4=\beta_5=\beta_6=0.\nonumber
\end{align}
The sixth Painlev\'e equation usually has four parameters, and one other
than $\beta_1$, $\beta_2$, and $\beta_3$ can be regarded as the integral
constant that appears when integrating the fourth-order equation into
the second-order equation.
Of the general equations of type (H) with 3 singular points,
the case of $\beta_4=0$ corresponds to the sixth Painlev\'e equation.

This condition $\beta_4=0$ corresponds to the case that the fourth-order
differential equation can be integrated into the second-order equation in an elementary manner.
We will now look at that calculation.

First, since the equation of type (H) is a homogeneous equation,
it can be reduced to a third-order equation.
By replacing the independent variable $t$ with $s$ such that $d/ds=t(t-1)d/dt$,
the equation can be written as
 \begin{align}
 0={\cal L}f\cdot f
 =&\left[
 {{\cal D}_s}^4+\left(t(t-1)+1-\beta_1-\beta_5t-\beta_6t^2
 -2(2t-1)\frac{d}{ds}\right){{\cal D}_s}^2
 \right.\\
 &\left. +t(t-1)\left(\beta_2t+\beta_3
 +(\beta_4+\beta_5+(t+1)\beta_6)\frac{d}{ds}
 +\frac{d^2}{ds^2}\right)
 \right] f\cdot f.\nonumber
\end{align}
Dividing the whole by $2f^2$ and setting
$h=\frac{df/ds}{f}=t(t-1)\frac{df/dt}{f}$, the equation can be written
as
\begin{align}
 0=&\frac{d^3h}{ds^3}+6\left(\frac{dh}{ds}\right)^2
 -2(2t-1)\left(\frac{d^2h}{ds^2}+2h\frac{dh}{ds}\right)\\
 &+(2t(t-1)+1-\beta_1-\beta_5t-\beta_6t^2)\frac{dh}{ds}+2t(t-1)h^2\nonumber\\
 &+t(t-1)(\beta_4+\beta_5+(t+1)\beta_6)h+\frac12 t(t-1)(\beta_2t+\beta_3).\nonumber
\end{align}
Here we used
\begin{align}
 &\frac{d^3h}{ds^3}+6\left(\frac{dh}{ds}\right)^2
 =\frac{{{\cal D}_s}^4f\cdot f}{2f^2},\quad
 \frac{d^2h}{ds^2}+2h\frac{dh}{ds}
 =\frac{\frac{d}{ds}{{\cal D}_s}^2f\cdot f}{2f^2},\quad
 \frac{dh}{ds}
 =\frac{{{\cal D}_s}^2f\cdot f}{2f^2}.
\end{align}
By replacing the variable $s$ with $t$ again,
the equation is written as
\begin{align}
 0
 =&t^2(t-1)^2\frac{d^3h}{dt^3}+(2t-1)t(t-1)\frac{d^2h}{dt^2}
 +6t(t-1)\left(\frac{dh}{dt}\right)^2
 -4(2t-1)h\frac{dh}{dt}+2h^2\\
 &-(\beta_1+t\beta_5+t^2\beta_6)\frac{dh}{dt}
 +(\beta_4+\beta_5+(t+1)\beta_6)h+\frac12 (\beta_2t+\beta_3).\nonumber
\end{align}
Here we used
\begin{align}
 &\frac{d}{ds}=t(t-1)\frac{d}{dt},\quad
 \frac{d^2}{ds^2}=t^2(t-1)^2\frac{d^2}{dt^2}+(2t-1)t(t-1)\frac{d}{dt},\\
 &\frac{d^3}{ds^3}=t^3(t-1)^3\frac{d^3}{dt^3}
 +3(2t-1)t^2(t-1)^2\frac{d^2}{dt^2}+(6t^2-6t+1)t(t-1)\frac{d}{dt}.\nonumber
\end{align}

In the gauge transformation, $2\beta_4+\beta_5+2\beta_6$ are invariants,
and when
\begin{align}
 2\beta_4+\beta_5+2\beta_6=0
\end{align}
further integration is possible and the result is a second-order differential equation.
Substituting $\beta_4=-\frac12\beta_5-\beta_6$,
and multiplying each side by $2(d^2h/dt^2)$, we can integrate it as
\begin{align}
 0=&t^2(t-1)^2\left(\frac{d^2h}{dt^2}\right)^2\!
 +4t(t-1)\left(\frac{dh}{dt}\right)^3\!
 -4(2t-1)h\left(\frac{dh}{dt}\right)^2\!
 +4h^2\frac{dh}{dt}\\
 &-\beta_1\left(\frac{dh}{dt}\right)^2\!
 +\beta_2\left(t\frac{dh}{dt}-h\right)
 +\beta_3\frac{dh}{dt}\nonumber\\
 &+\frac{\beta_5}{2}\left(h\frac{dh}{dt}-t\left(\frac{dh}{dt}\right)^2\right)
 +\frac{\beta_6}{2}\left(2th\frac{dh}{dt}-t^2\left(\frac{dh}{dt}\right)^2-h^2\right)+C. \nonumber
\end{align}
This expression coincides with the differential equation satisfied by the Hamiltonian of the sixth Painlev\'e equation when normalized to $\beta_5=\beta_6=0$ by gauge transformation.

\bigskip
\noindent
{\it Acknowledgments} \quad
We would like to thank Prof.\ Iwaki, Prof.\ Ohyama, Prof.\ Nagoya and Prof.\ Jimbo for useful discussions.
This work was supported by JSPS KAKENHI Grant Number JP22K03348.

\end{document}